\documentclass{amsart}
\usepackage{amsmath, amssymb, mathrsfs, amsthm, xifthen, verbatim, color, mathrsfs}

\usepackage[arrow, matrix]{xy}

\definecolor{DarkBlue}{rgb}{0,0.2,0.6}
\definecolor{PinkPurple}{rgb}{0.1,0.1,0.8}
\definecolor{Green}{rgb}{0,1,0}

\newcommand{\ma}[1]{{{#1}}}
\newcommand{\me}[1]{{{#1}}}

\usepackage[pdftex,
            pdfauthor={Mahmood Alaghmandan, Mehdi Ghasemi},
            pdftitle={Spectral realization of commutative algebras},
            pdfsubject={Seminormed Algebras},
            pdfkeywords={Compactification, Norms, Seminorms, Topology, Ultrafilters},
            pdfproducer={Latex with hyperref},
            pdfcreator={PDFLaTeX},
            linkcolor=DarkBlue,
            citecolor=PinkPurple,
            colorlinks=true]{hyperref}

\newtheorem{thm}{Theorem}[section]
\newtheorem{lemma}[thm]{Lemma}
\newtheorem{prop}[thm]{Proposition}
\newtheorem{crl}[thm]{Corollary}

\theoremstyle{definition}
\newtheorem{dfn}[thm]{Definition}
\newtheorem{exm}[thm]{Example}

\newtheorem{rem}[thm]{Remark}

\newcommand{\reals}{\mathbb{R}}
\newcommand{\naturals}{\mathbb{N}}

\newcommand{\cplx}{\mathbb{C}}

\newcommand{\fld}{\mathbb{F}}

\newcommand{\X}[1]{\mathcal{X}(#1)}

\newcommand{\supp}[1]{\text{supp}(#1)}

\newcommand{\norm}[2]{\|\ifthenelse{\isempty{#2}}{\cdot}{#2}\|_{#1}}
\newcommand{\cl}[2]{\overline{#2}^{\ifthenelse{\isempty{#1}}{}{#1}}}
\newcommand{\Cnt}[2]{\mathrm{C}_{\ifthenelse{\isempty{#1}}{}{#1}}(#2)}
\newcommand{\Meas}[2]{\mathrm{M}_{\ifthenelse{\isempty{#1}}{}{#1}}(#2)}
\newcommand{\Psd}[2]{\mbox{Psd}_{\ifthenelse{\isempty{#1}}{}{#1}}(#2)}
\newcommand{\Bnd}[2]{\mbox{B}_{\ifthenelse{\isempty{#1}}{}{#1}}(#2)}
\newcommand{\Sp}[2]{\mathfrak{sp}_{\ifthenelse{\isempty{#1}}{}{#1}}(#2)}
\newcommand{\Top}[2]{\tau(#1{\ifthenelse{\isempty{#2}}{}{,#2}})}
\newcommand{\nbhd}[2]{\mathcal{N}_{\ifthenelse{\isempty{#2}}{}{#2}}(#1)}
\newcommand{\map}[3]{#1:#2\longrightarrow #3}

\newcommand{\ignore}[1]{{ }}

\begin{document}

\title[Seminormed $\ast$-subalgebras of $\ell^{\infty}(X)$]{Seminormed $\ast$-subalgebras of $\ell^{\infty}(X)$}
\author[M. Alaghmandan, M. Ghasemi]{Mahmood Alaghmandan$^1$, Mehdi Ghasemi$^2$}
\address{$^1$Department of Mathematical Sciences,\newline\indent
Chalmers University of Technology and\newline\indent
University of Gothenburg,\newline\indent
Gothenburg SE-412 96, Sweden}
\email{mahala@chalmers.se}
\address{$^2$Department of Mathematics and Statistics,\newline\indent
University of Saskatchewan,\newline\indent
106 Wiggins Road,\newline\indent
Saskatoon, SK. Canada S7N 5E6}
\email{mehdi.ghasemi@usask.ca}
\keywords{
measurable functions, commutative normed algebras, function algebras, Gelfand spectrum, seminormed algebras, measures on Boolean rings}
\subjclass[2010]{Primary 46J05, 46J10,
Secondary 46E30.}
\date{\today}
\begin{abstract}
Arbitrary representations of a commutative unital ($\ast$-) $\fld$-algebra $A$ as a subalgeba of $\fld^X$ are considered, where $\fld=\cplx$ 
or $\reals$ and $X\neq\emptyset$. The  Gelfand spectrum of $A$ is explained as a topological extension of $X$ where a seminorm on the image of
$A$ in $\fld^X$ is present. It is shown that among all seminormes, the $\sup$-norm is of special importance which reduces $\fld^X$ to 
$\ell^{\infty}(X)$. The Banach subalgebra of $\ell^{\infty}(X)$ of all $\Sigma$-measurable bounded functions on $X$,  $\Meas{b}{X,\Sigma}$, is studied for
which $\Sigma$ is a $\sigma$-algebra of subsets of $X$. In particular, we study   lifting of positive measures from $(X, \Sigma)$  to the Gelfand  spectrum of $\Meas{b}{X,\Sigma}$ and observe an unexpected shift in the support of measures. In the case that $\Sigma$ is the 
Borel algebra of a topology, we study the relation of the underlying topology of $X$   and  the one of the Gelfand   spectrum of $\Meas{b}{X, \Sigma}$.

\end{abstract}
\maketitle
\section{Introduction}
It is  common to look at rings and algebras as families of functions over a nonempty set with values in a suitable ring or field.
This is specially helpful if one wants to study the ideal structure of a ring or algebra which naturally involves topological 
notions, mainly compactness.

In this article, we try to summarize some  observations about topological algebras in an abstract manner. One motivation comes from
\cite{MGHSKMM2} which attempts to represent positive linear functionals on a given commutative unital algebra as an integral with respect to a 
positive measure on the space of characters of the algebra. This is done by realizing the algebra as a subalgebra of continuous functions 
over the character space.

During the present article we always assume that $A$ is a commutative algebra over the field $\fld=\reals$ or $\cplx$ equipped with a 
seminorm $\rho$. In section \ref{gen-sett}, first we provide a brief overview of the theory of seminormed algebras and their Gelfand spectrum.
Then, we assume that $A$ can be embedded into $(\fld^X, \rho)$ for a nonemty set $X$ where $\rho$ defines a submultiplicative 
seminorm on a subalgebra of $\fld^X$ including the image of $A$. This induces a seminormed structure on $A$ as well. Theorem \ref{jst-bndd} gives
a necessary and sufficient condition for $X$ to be dense in the Gelfand spectrum of $A$,  {that is,  when the topology induced by the seminorm 
is equivalent to the topology induced by the norm infinity}. 

Motivated by \cite{MGHSKMM2}, where  positive linear functionals on an algebra are presented as an integral with
respect to a constructibly Radon measure, in section \ref{meas-struct}, we consider a measurable structure $\Sigma$ on $X$ and study the spectrum 
of the algebra of bounded measurable functions  on $(X, \Sigma)$, denoted by $\Meas{b}{X,\Sigma}$. We prove that positive measures on $X$ lift 
to positive measures over the spectrum of $\Meas{b}{X,\Sigma}$, but this lifting shifts the support of the original measure out of $X$ modulo 
at most a countable subset of $X$ (propositions \ref{meas-ext} and \ref{meas-shift}). Then we choose $\Sigma$ to be $\mathcal{B}_{\tau}$, the 
Borel algebra of a topology $\tau$ on $X$, and observe some connections between $\tau$ and the spectrum of $\Meas{b}{X,\mathcal{B}_{\tau}}$ 
(proposition \ref{Open-halos} and theorem \ref{semi-robinson}).
\subsection*{Notations}
Let $X$ be a non-empty set  and $\mathfrak{S}$ be a structure on $X$ which induces a topology on $X$.  We denote this topology by 
$\Top{X}{\mathfrak{S}}$. For instance, let $\mathfrak{S}$ be a family of functions, defined on$X$, with values in a topological space
Then $\Top{X}{\mathfrak{S}}$ is the coarsest topology on $X$ which makes every function in $\mathfrak{S}$ continuous.

The $\sigma$-algebra of sets induced on $X$ by a set $\Lambda\subseteq2^X$ is denoted by $\sigma(\Lambda)$. In particular if $\tau$ is a topology
on $X$ then $\sigma(\tau)$ is the $\sigma$-algebra, denoted by  $\mathcal{B}_{\tau}$, of all Borel subsets of $(X,\tau)$.
\section{{Involutive subalgebras of $\ell^{\infty}(X)$}}\label{gen-sett}
\ma{The  set theory which is applied in this paper is  ZFC}. Throughout this article all algebras are assumed to be involutive and commutative over a field $\fld$ 
(which is either $\reals$ or $\cplx$ as specified). \ma{Subsequently, all  ($*$-)algebra homomorphisms are also supposed to be $\fld$-module maps.}
\ma{
\begin{dfn}\label{d:semi-norm}
A function $\map{\rho}{A}{[0,\infty]}$ is called a {\it quasi-norm } on $A$ if 
\begin{itemize}
	\item[(1)]{
	$\rho(a^*) = \rho(a)$  $\forall a\in A$
	}
	\item[(2)]{
	$\rho(a+b)\leq\rho(a) + \rho(b)$ $\forall a,b\in A$ (subadditive),
	}
	\item[(3)]{
	$\rho(ra)=|r|\rho(a)$ $\forall r\in\fld~\forall a\in A$,
	}
\end{itemize}
it is called {\it submultiplicative}, if
\begin{itemize}
	\item[(4)]{
	$\forall a,b\in A\quad \rho(a\cdot b)\leq\rho(a)\rho(b)$.
	}
\end{itemize}
A {\it quasi-norm } $\rho$ on $A$ is called a {\it seminorm} if  $\rho(a)<\infty$  for every   $a \in A$.
\end{dfn}
}

Let $A$ be a commutative algebra and  let $\rho$ be a quasi-norm on $A$. The set of all elements of $A$ with a finite quasi-norm $\rho$ is denoted by $\Bnd{\rho}{A}$, i.e.,
\[
	\Bnd{\rho}{A} = \{a\in A : \rho(a)<\infty\}.
\]
If $\rho$ is a submultiplicative quasi-norm $\rho$, it is clear that $\Bnd{\rho}{A}$ is a $\ast$-subalgebra of $A$ and the restriction of 
$\rho$ to $\Bnd{\rho}{A}$ is a {seminorm}. An element $a\in A$ is called symmetric if $a^*=a$. 
The set of all symmetric elements of $A$ are denoted by $S(A)$. 
An algebra $A$ with a seminorm $\rho$ forms a {\it seminormed algebra} if $\rho$ is   submultiplicative.
For a seminormed algebra $(A, \rho)$, the set of all \ma{non-zero} $\ast$-algebra homomorphisms $\map{\alpha}{A}{\fld}$ is 
denoted by $\X{A}$. The subset of $\X{A}$ consisting of all $\rho$-continuous homomorphisms is called the {\it Gelfand spectrum}
of $(A,\rho)$ and denoted by $\Sp{\rho}{A}$ which is a closed subspace of $\X{A}$.

Note that $S(A)$ is always an $\reals$-algebra and if $\fld = \cplx$, then $A=S(A)\oplus iS(A)$. Moreover there is a 
bijective correspondence between $\X{S(A)}$ and $\X{A}$: If $\alpha\in\X{A}$, then clearly $\alpha|_{S(A)}$ is real 
valued and hence $\alpha|_{S(A)}\in\X{S(A)}$. Also for every $\alpha\in\X{S(A)}$, its extension defined by
$\bar{\alpha}(a+ib)=\alpha(a)+i\alpha(b)$ is well-defined and $\bar{\alpha}|_{S(A)}=\alpha$. Next, we give a 
characterization of all $\rho$-continuous $\fld$-valued homomorphisms. \ma{The following lemma was proved as Lemma~3.2 in \cite{GKM}. }

\begin{lemma}\label{cts-hom}
Let $\alpha \in \X{A}$. Then   $\alpha\in\Sp{\rho}{A}$ if and only if $|\alpha(a)|\leq\rho(a)$, for all $a\in A$. 
\end{lemma}

The Gelfand spectrum $\Sp{\rho}{A}$ (as well as $\X{A}$) naturally carries a Hausdorff topology as a subspace of 
$\fld^A$ with the product topology. For a real number $r>0$, let $D_{r}:=\{c\in\fld : |c|\leq r\}$. According to Lemma ~\ref{cts-hom}, 
 $\Sp{\rho}{A} \subseteq \prod_{a\in A}D_{\rho(a)}$.  \ma{One simple approximation argument implies that every element in the closure of  $\Sp{\rho}{A}$ is a  $\ast$-algebra $\fld$-homomorphism. But it also belongs to  $\prod_{a\in A}D_{\rho(a)}$. Therefore, the closure of $\Sp{\rho}{A}$ is a subset of  $\Sp{\rho}{A} \cup \{{\bf 0}\}$ where $\bf 0$ is the constant linear functional zero on $A$.  From now on, we use $\Sp{\rho}{A}$ to denote it as a topological subspace of $\prod_{a\in A}D_{\rho(a)}$. }

\vskip0.5em

\ma{ Note that the difference between the following corollary and \cite[Corollary~3.3]{GKM} is due to the fact that we exclude zero in the definition of  $\X{A}$. }
\ma{
\begin{crl}\label{unit-cmpt}
Let $(A, \rho)$ be a commutative seminormed algebra. If $A$ is unital then $\Sp{\rho}{A}$ is compact. If $\Sp{\rho}{A}$ is compact then there exists an element $a_0 \in A$ such that $|\alpha(a_0)| \geq 1$ for every $\alpha \in \Sp{\rho}{A}$.
\end{crl}
}
\ma{
\begin{proof}
\ma{If $A$ is unital, one may use the identity element, $\bf 1$, (for which we have $\alpha({\bf 1})=1$ for every $\alpha \in \Sp{\rho}{A}$) to show that   
$\bf 0$ does not belong to the closure of $\Sp{\rho}{A}$. Therefore, $\Sp{\rho}{A}$ is indeed a closed set in $\prod_{a\in A}D_{\rho(a)}$. 
And subsequently,  $\Sp{\rho}{A}$ is  compact. }

Now suppose that $\Sp{\rho}{A}$ is compact. 
Therefore, $\Sp{\rho}{A}$ is a closed subset of $\prod_{a\in A}D_{\rho(a)}$\me{, not containing {\bf 0}. So, there exists $a\in A$ and 
$\epsilon>0$ such that the projection on $a^{th}$ component of the above product does not intersect with the neighbourhood 
$(-\epsilon, \epsilon)$ of $0$. Thus, this particular element $a$, satisfies $|\alpha(a)|\ge\epsilon$ for each $\alpha\in\Sp{\rho}{A}$.} 
Let $k = \inf\{ |\alpha(a)|: \alpha \in \Sp{\rho}{A}\} \geq \epsilon$ and $a_0:= a/k$. The claim follows for $a_0$. 
\end{proof}
}

\ma{
\begin{rem}\label{r:non-unital}
Every commutative seminormed algebra $(A, \rho)$  can be embedded into the unital algebra $A_1 := A\oplus\fld$ with multiplication 
$(a,\lambda)(b,\gamma)=(ab+\gamma a+\lambda b,\lambda\gamma)$. Defining $\rho_1(a+\lambda)=\rho(a)+|\lambda|$ we also 
obtain a seminorm on $A_1$ which makes the natural embedding $a\mapsto (a,0)$ continuous.  For each $\alpha \in \X{A}$, define the extension $\alpha'(a,\lambda)= \alpha(a) + \lambda$ which is obviously an element in $\X{A_1}$. So one can regard $\X{A}$ as a subset of $\X{A_1}$. 
Regarding $\fld$ as a commutative algebra, we know that $\X{\fld}$ has only one element which is  the identity map. This leads to the fact that  $\X{A_1}\setminus\X{A}$ consists 
of exactly one element which is zero on $A$ and maps $(a,\lambda)$ to $\lambda$ (denoted by $\hat{\infty}$). Clearly,
$\hat{\infty}\in\Sp{\rho_1}{A_1}$, therefore $A$ is a closed maximal ideal of $A_1$. 
 Moreover, if $\Sp{\rho}{A}$ is not compact,  $\Sp{\rho_1}{A_1}$ is the one-point compactification of $\Sp{\rho}{A}$.
\end{rem}
}
\vskip0.5em

\ma{
Every element $a\in A$ induces a map $\map{\hat{a}}{\X{A}}{\fld}$ defined by $\hat{a}(\alpha) := \alpha(a)$ for each $\alpha \in \X{A}$. Every $*$-algebra homomorphism $\map{\phi}{A}{B}$ induces a mapping $\map{\phi_*}{\X{B}}{\X{A}} \cup\{\bf 0\}$ defined by $\phi_*(\beta)=\beta \circ \phi$ for each $\beta \in  \X{B}$.
}

\ignore{
Every element $a\in A$ induces a map $\map{\hat{a}}{\X{A}}{\fld}$ defined by 
\[
\begin{array}{lccl}
	\hat{a} : & \X{A} & \longrightarrow & \fld\\
	 & \alpha & \mapsto & \alpha(a),
\end{array}
\]
which is the projection into the $a^{\rm th}$ component of $\fld^A$ that is continuous. Therefore we always have an
algebra homomorphism $\map{\hat{ }}{A}{\Cnt{}{\X{A}}}$ {where $\Cnt{}{\X{A}}$ denotes the set of all $\fld$-valued continuous functions on
$\X{A}$ where $\X{A} \cup \{0\}$ is equipped with the Alexandroff topology}. If $\phi$ is surjective then $\phi_*$ is a topological embedding.
}

\ignore{
\begin{rem}\label{one-point-cmptfctn}
Let $A$ and $B$ be $\fld$-algebras. Every $\ast$-algebra homomorphism $\map{\phi}{A}{B}$ induces a continuous map 
\[
\begin{array}{lccl}
	\phi_* : & \X{B} & \longrightarrow & \X{A}\cup\{{\bf0}\}\\
	 & \beta & \mapsto & \beta\circ\phi,
\end{array}
\]
where ${\bf0}$ is the zero map on $A$ and $\X{A}\cup\{{\bf0}\}$ is equipped with the {\it Alexandroff topology} \cite[3.15]{Gil-Jer}.
If $\phi$ is surjective then $\phi_*$ is a topological embedding.
\end{rem}
}

Suppose that $B$ is equipped with a seminorm $\rho$. The homomorphism $\phi$ induces a seminorm $\rho_{\phi}$ on $A$ 
defined by $\rho_{\phi}(a) = \rho(\phi(a))$. If $\rho$ is submultiplicative, then so is $\rho_{\phi}$.
The map $\phi$ as a homomorphism between seminormed algebras $(A, \rho_{\phi})$ and $(B, \rho)$ is continuous. Therefore
$\phi_*$ maps $\Sp{\rho}{B}$ continuously into $\Sp{\rho_{\phi}}{A}$.
Here  we are mainly interested in the case where $B$ is a subalgebra of $\fld^X$ for a non-empty set $X$. This generally enables 
us to realize $\Sp{}{A}$ relative to $X$.

Let $\rho$ be a submultiplicative quasi-norm on $\fld^X$ with $\rho(1)\geq1$. There is a natural map $\map{e}{X}{\X{\fld^X}}$
which, to every $x\in X$, assigns the \textit{evaluation map} $\map{e_x}{\fld^X}{\fld}$, defined by $e_x(f):=f(x)$. 
This is clear that $e_x\in\X{\fld^X}$. We denote the set of all $\rho$-continuous evaluations by $X_{\rho} (=e X\cap\Sp{\rho}{\Bnd{\rho}{X}})$.

Let  $\map{\iota}{A}{\Bnd{\rho}{\fld^X}}$ \ma{be} an algebra homomorphism. By abuse of notation, we use  $\iota_*$  to  denote  the  induced map $\map{\iota_*|_X}{(X,\tau)}{\Sp{\rho}{A}}$.

\begin{thm}\label{jst-bndd}
Let  $(A,\rho)$ be a commutative seminormed algebra and let $\map{\iota}{A}{\Bnd{\rho}{\fld^X}}$ be a homomorphism. Define $\rho_{\iota}$ on $A$ as before.
Then $\iota_*(X_{\rho})$ is dense in $\Sp{\rho_{\iota}}{A}$ if and only if there exists $D>0$ such that
\[
	\rho_{\iota}(a)\leq D\cdot\sup_{x\in X_{\rho}} |e_x(\iota a)|,
\]
for all $a\in A$.
\end{thm}
\begin{proof}
\ma{Let $\Cnt{}{\Sp{\rho_{\iota}}{A}}$ denote the space of all continuous functions on $\Sp{\rho_{\iota}}{A}$.}
The Gelfand map $\map{\hat{ }}{A}{\Cnt{}{\Sp{\rho_{\iota}}{A}}}$ is continuous and hence for some $C>0$,
\[
	\sup_{\beta\in\Sp{\rho_{\iota}}{A}}|\hat{a}(\beta)|\leq C\cdot\rho_{\iota}(a),\quad\forall a\in A.
\]
\begin{itemize}
\item[($\Rightarrow$)]{
Since $\iota_*(X_{\rho_{\iota}})$ is dense in $\Sp{\rho_{\iota}}{A}$ we have
\[
	\sup_{x\in X_{\rho}}|e_{x}(\iota a)| = \sup_{\beta\in\Sp{\rho_{\iota}}{A}}|\beta(a)|,
\]
and it suffices to take $D=C$.
}
\item[($\Leftarrow$)]{
In contrary, suppose that $\alpha\in\Sp{\rho_{\iota}}{A}\setminus\cl{}{\iota_*(X_{\rho})}\neq\emptyset$.
There exists $f\in\Cnt{}{\Sp{\rho_{\iota}}{A}}$ such that $f(\alpha)=1$ and $f|_{\iota_*(X_{\rho})}=0$.
Since $\hat{A}$ separates \ma{points} of $\Sp{\rho_{\iota}}{A}$, $\hat{a}\in\Cnt{}{\Sp{\rho_{\iota}}{A}}$, and 
$\Sp{\rho_{\iota}}{A}$ is compact, by Stone-Weierstrass theorem, it is dense in $\Cnt{}{\Sp{\rho_{\iota}}{A}}$.
Therefore, for $\epsilon>0$, there is $a_{\epsilon}\in A$ with $\norm{}{f-\hat{a_{\epsilon}}}<\epsilon$.
Take an $\epsilon>0$ such that $\frac{1-\epsilon}{\epsilon}>CD$. 
Then $|f(\alpha)-\alpha(a_{\epsilon})|=|1-\alpha(a_{\epsilon})|<\epsilon$ or $1-\epsilon<|\alpha(a_{\epsilon})|<1+\epsilon$.
Also $|f(\iota_*e_{x})-e_{x}(\iota a_{\epsilon})|=|0-\iota a_{\epsilon}(x)|<\epsilon$ for all $x\in X_{\rho}$. Now
\[
\begin{array}{lcl}
\sup_{\beta\in\Sp{\rho_{\iota}}{A}}|\beta(a_{\epsilon})| & \leq & C\rho_{\iota}(a_{\epsilon})\\
 & \leq & CD\sup_{x\in X_{\rho}}|e_{x}(\iota a)|\\
 & \leq & CD\epsilon\\
 & < & 1-\epsilon,
\end{array}
\]
and hence $|\alpha(a_{\epsilon})|<1-\epsilon$, a contradiction.
}
\end{itemize}
\end{proof}
The clear implication of Theorem \ref{jst-bndd} is that if one is   to realise a unital commutative algebra 
as a subalgebra of $(\fld^X,\rho)$ the natural choice for $\rho$ is the $\sup$-norm over $X$ which is defined by
\begin{equation}\label{eq:sup-norm}
	\norm{X}{f} = \sup_{x\in X}|f(x)|.
\end{equation}
We denote $\Bnd{\norm{X}{.}}{\fld^X}$ by $\ell^{\infty}(X)$. According to Theorem \ref{jst-bndd} the image of $X$ under the 
map $x\mapsto e_x$ is dense in $\Sp{\|\cdot\|_X}{\ell^{\infty}(X)}$ and also for $\map{\iota}{A}{\ell^{\infty}(X)}$,  we have 
$\cl{\norm{X\iota}{}}{\iota_*(X_{\norm{X}{}})}=\Sp{\norm{X\iota}{}}{A}$. 

Suppose that $\tau$ is a topology on $X$ and denote the set of all 
$\tau$-continuous bounded $\fld$-valued functions on $X$ by $\Cnt{b}{X,\tau}$ or simply by $\Cnt{b}{X}$, if there is no 
risk of confusion.
Let  $\map{\iota}{A}{\ell^{\infty}(X)}$ an algebra homomorphism.
Then one can  show that the  induced map $\map{\iota_*|_X}{(X,\tau)}{\Sp{\norm{X\iota}{}}{A}}$ is continuous if and only if 
$\iota A\subseteq\Cnt{b}{X}$. 

It is well known that if $(X,\tau)$ is a completely regular Hausdorff space, then $\Sp{\norm{X\iota}{}}{\Cnt{b}{X}}$ is the
Stone-\v{C}ech compactification of $(X,\tau)$. Moreover, every Hausdorff compactification of $(X,\tau)$ is homeomorphic to the
spectrum of a unital subalgebra of $\Cnt{b}{X}$. In the next section we study the algebra of bounded measurable functions
for a measurable structure on $X$.

\section{Measurable structures on $X$}\label{meas-struct}
Let $\Sigma$ be a $\sigma$-algebra of subsets of $X$. 
Let $\Meas{b}{X,\Sigma}$ be the algebra of all bounded $\Sigma$-measurable functions on $(X, \Sigma)$. 
Suppose that $\Meas{b}{X,\Sigma}$ separates points of $X$. Hence $X$ sits inside \ma{$\Sp{\|\cdot\|_X}{\Meas{b}{X,\Sigma}}$},  injectively and its image is dense. 
Although we are assuming that $\Meas{b}{X,\Sigma}$ separates points of $X$, this does not imply that $\Sigma$ contains singletons \ma{as we see in the following example.} 
\begin{exm}
Recall that a topological space $(X,\tau)$ is called a $T_0$ space, if for every $x\neq y\in X$, either $x\not\in\cl{\tau}{\{y\}}$ or
$y\not\in\cl{\tau}{\{x\}}$. Then characteristic functions of open sets clearly separate points of $X$. Let $\omega_1$ be the first
uncountable ordinal and $X=\omega_1+1$. The family of sets $R_a:=\{x\in X: x>a\}$ form a basis for a topology on $X$. This topology is
evidently $T_0$ and satisfies first axiom of countability at every point except $\omega_1$. 
Although \ma{$\{\omega_1\}= \bigcap_{\omega_1>a}R_a$},  \ma{every countable intersection of sets $R_a$ for $a < \omega_1$} contains ordinals smaller 
than $\omega_1$. Thus $\{\omega_1\}$ does not belong to   the $\sigma$-algebra generated by 
$\{R_a:a\in X\}$, denoted by $\Sigma_r$, while $\Meas{b}{X,\Sigma_r}$ separates points of $X$. Note that the topology of $\omega_1+1$ in 
this case \ma{is} not first countable.  \ma{Singletons  always  belong to       the $\sigma$-algebra of Borel subsets of  first countable spaces. }
\end{exm}
We denote \ma{$\Sp{\| \cdot\|_X}{\Meas{b}{X,\Sigma}}$} by $\xi_{\Sigma}X$. \ma{Note that since $\Meas{b}{X, \Sigma}$ separates the points of $X$,
\me{we can think of $\xi_{\Sigma}X$ as a compactification of $X$,}
i.e., it is a compact Hausdorff space  such that there is an injection $\map{\psi}{X}{\xi_{\Sigma}X}$ such that $\psi(X)$ is a dense subspace of $\xi_{\Sigma}X$. Further,  for every bounded $\Sigma$-measurable function $f$ on $X$, the function $f \circ \psi^{-1}$ is continuously extendible over $\xi_{\Sigma}X$. Also, $\xi_{\Sigma}X$  is unique  (up to a homeomorphism)  with this property in this sense that for every other   compactification of $X$, say $\gamma X$, on which elements of $\Meas{b}{X,\Sigma}$ are continuously extendible, there is 
a  continuous map $\map{\iota}{\gamma X}{\xi_{\Sigma}X}$ agreeing on the images of $X$ in $\xi_\Sigma X$ and $\gamma X$. }

Let $\chi_E$ be the characteristic function of $E$ as on $X$ for $E\in\Sigma$. Denoting its continuous extension to $\xi_{\Sigma}X$ 
with $\tilde{\chi}_E$ we have:
\[
	\tilde{\chi}_E^2=\tilde{\chi_E^2}=\tilde{\chi}_E;
\]
thus it ranges over $\{0,1\}$, which implies that $\tilde{\chi}_E$ itself must be the characteristic function of a set, say
$\tilde{E}$ in $\xi_{\Sigma}X$.	
\begin{lemma}\label{img-of-meas}
Let $E \in \Sigma$. Then $\cl{}{E}=\tilde{E}$  where $\cl{}{E}$ is the closure of $E$ in $\xi_\Sigma X$.
\end{lemma}
\begin{proof}
It is clear that $\tilde{E}=\tilde{\chi}_E^{-1}(\{1\})$ is closed and $E\subseteq\tilde{E}$. Thus $\cl{}{E}\subseteq\tilde{E}$.
\ignore{
Let $\mathcal{F}$ be the ultrafilter containing neighbourhoods of $z\in\tilde{E}$ in $\xi_{\Sigma}X$ \cite[\S12]{willard}. 
If $E\not\in\mathcal{F}$,
}
\ma{ If $z \notin \overline{E}$,  then for an open neighbourhood $U$ of $z$ we have $U\cap E=\emptyset$.}  Therefore there is a function 
$f\in\Meas{b}{X,\Sigma}$ and an open interval $I$ in $\reals$ such that $z\in\tilde{f}^{-1}(I)\subseteq U$. Let $F=f^{-1}(I)\in\Sigma$,
then $E\cap F=\emptyset$, so $\chi_E\cdot\chi_F=0$ and $\tilde{\chi}_E\cdot\tilde{\chi}_F=0$. Since $\tilde{\chi}_F(z)=1$ the later
equation implies $\tilde{\chi}_E(z)=0$. This contradicts assumption $z\in\tilde{E}$, therefore $\tilde{E}=\cl{}{E}$.
\end{proof}
Using the above lemma, we investigate some properties of $X$ as a subspace of $\xi_{\Sigma}X$.
\begin{crl}\label{MCmpt}~
\begin{enumerate}
	\item{\label{MCmptI} $\cl{}{E}$ is a clopen subset of $\xi_{\Sigma}X$ for every $E\in\Sigma$;}
	\item{\label{MCmptIII} $\tilde{\Sigma}:=\{\tilde{E}:E\in\Sigma\}$ \ma{forms} a basis for the topology of $\xi_{\Sigma}X$;}
\end{enumerate}
In addition, if $\Sigma$ contains all singletons, then
\begin{enumerate}
	\setcounter{enumi}{2}
	\item{\label{MCmptII} $X$ is an open dense subspace of $\xi_{\Sigma}X$ whose subspace topology is discrete;}
	\item{\label{MCmptIV} For a subset $Y\subset X$, $\cl{}{Y}=Y$ if and only if $Y$ is finite.}
\end{enumerate}
\end{crl}
\begin{proof}
\eqref{MCmptI} Since $\cl{}{E}=\tilde{E}=\tilde{\chi}_{E}^{-1}(\{1\})=\tilde{\chi}_{E}^{-1}(\frac{1}{2},\infty)$ and $\tilde{\chi}_E$ 
is continuous, we conclude  that $\tilde{E}$ is clopen.

\eqref{MCmptIII} The family $\{\tilde{f}^{-1}[0,1]: f\in\Meas{b}{X,\Sigma}\}$ forms a basis for the closed sets of $\xi_{\Sigma}X$.
Note that $E=f^{-1}[0,1]\in\Sigma$ and $\tilde{E}=\cl{}{E}=\tilde{f}^{-1}[0,1]$ which is clopen by \eqref{MCmptI} and the conclusion follows.

\eqref{MCmptII} By \eqref{MCmptI}, the closure of every element of $\Sigma$ is open in
$\xi_{\Sigma}X$. Since the topology of $\xi_{\Sigma}X$ is Hausdorff and   $\Sigma$ contains all singletons, singletons are closed. Therefore $\{x\}$ is a clopen for every $x\in X$
and hence $X$ is open in $\xi_{\Sigma}X$, dense by Theorem \ref{jst-bndd} and the subspace topology on $X$ is discrete.

\eqref{MCmptIV} If $Y$ is finite, then since the topology of $\xi_{\Sigma}X$ is Hausdorff, it is also closed. Let $Y$ be an arbitrary 
subset of $X$. The set $\cl{}{Y}\subseteq\xi_{\Sigma}X$ is compact. If $Y=\cl{}{Y}$, then $\{\{x\}:~x\in Y\}$ is an open cover of $Y$ 
which will not have a finite subcover, if $Y$ is infinite.
\end{proof}

\begin{rem}\label{r:xi-Sigma-not-metrizable}
Let $ \Sigma $ be a $\sigma$-algebra of subsets of an infinite set $X$.  If there are infinitely many disjoint sets in $\Sigma$, 
then $\Meas{b}{X,\Sigma}$ is not separable. The proof is similar to the classical proof of the fact that $\ell^\infty(\naturals)$ is not separable. 
Hence, in this case $\xi_{\Sigma}X$ is not metrizable. (It is classically known that for a compact space $X$, $\Cnt{0}{X}$ is separable if and only if $X$ is metrizable (\cite[Theorem 2.4]{chou}). )
\end{rem}

By Lemma \ref{MCmpt}.\eqref{MCmptIII}, $\xi_{\Sigma}X$ is  {totally disconnected}. One can prove that $\{\tilde{E}:E\in\Sigma\}$ 
contains all clopen subsets of $\xi_{\Sigma}X$. To see this suppose that $Y\subseteq\xi_{\Sigma}X$ is clopen. Since $\xi_{\Sigma}X$ is compact, 
so is $Y$. By \ref{MCmpt}.\eqref{MCmptIII}, $Y=\bigcup_{i\in I}\tilde{E}_i$ as an open set, for a family $\{E_i\}_{i\in I}\subset\Sigma$.
Therefore $Y=\tilde{E}_{i_1}\cup\dots\cup\tilde{E}_{i_n}$ for $i_1,\dots,i_n\in I$, which belongs to $\tilde{\Sigma}$.

A topological space is called \textit{extremely disconnected} if the closure of every open set is open. The following shows when 
$\xi_{\Sigma}X$ is extremely disconnected.  For the relation between Boolean algebras and extremely disconnected spaces see 
\cite[\S 3.5]{Johnstone} or \cite[22.4]{Sikorski}. Commutative algebras with extremely disconnected Gelfand spectrum are forming the 
commutative class of AW$^*$-algebras where $\fld=\cplx$.

An algebra of sets is said to be \textit{complete} if it is closed under arbitrary union and hence intersection

\begin{prop}
If $\xi_{\Sigma}X$ is extremely disconnected, then $\tilde{\Sigma}$ is complete.
Conversely, if $\Sigma$ is complete, then $\xi_{\Sigma}X$ is extremely disconnected.
\end{prop}
\begin{proof}
Suppose that $\xi_{\Sigma}X$ is extremely disconnected and let $\Delta\subseteq\tilde{\Sigma}$. Then $U=\cup\Delta$ is open
and hence $\cl{}{U}$ is also open, thus belongs to $\tilde{\Sigma}$, say $\cl{}{U}=\tilde{E}$. 
If $\tilde{E}\setminus\cup\Delta\neq\emptyset$, then there exists is a clopen set $\emptyset\neq\tilde{F}\subseteq\tilde{E}\setminus\cup\Delta$.
Therefore $\cl{}{U}\subseteq\tilde{E}\setminus\tilde{F}\subsetneq\tilde{E}$, a contradiction. 

Now, suppose that $\Sigma$ is complete and let $U$ be an open set in $\xi_{\Sigma}X$. Take $\Delta\subset\Sigma$ such that
$U=\cup\tilde{\Delta}$. Since $\Sigma$ is complete, $E=\cup\Delta\in\Sigma$ and $\cl{}{U}\subseteq\cl{}{E}=\tilde{E}$.
If $\tilde{E}\setminus\cl{}{U}\neq\emptyset$ (so open) then it contains a nonempty clopen $\tilde{F}\in\tilde{\Sigma}$.
Now $\tilde{E}\setminus\tilde{F}$ is a clopen containing $\cup\tilde{\Delta}$ and strictly contained in $\tilde{E}$, a contradiction.
Thus $\cl{}{U}=\tilde{E}$ is clopen and hence $\xi_{\Sigma}X$ is extremely disconnected.
\end{proof}
\ma{To prove the last proposition of this subsection, we need the following well-known lemma. The proof is  straightforward, so  we omit it here.}

\begin{lemma}\label{cntbly-gen}
Let $(X,\tau)$ be a second countable topological space and let $\mathcal{B}_{\tau}$ be the $\sigma$-algebra of 
Borel subsets of $X$. Then $\mathcal{B}_{\tau}$ is countably generated.
\end{lemma}
\ignore{
\begin{proof}
Let $\mathcal{U}$ be a countable base for $\tau$. Then $\tau\subseteq\sigma(\mathcal{U})$  and hence, 
$\mathcal{B}_{\tau}=\sigma(\tau)=\sigma(\mathcal{U})$.
\end{proof}
}
\begin{prop}\label{meas-fuc-alg}
Suppose that $\Sigma$ is a countably generated $\sigma$-algebra on $X\neq\emptyset$ such that every open subset of $\xi_{\Sigma}X$ 
belongs to $\sigma(\tilde{\Sigma})$. Let $\map{\iota}{A}{\ell^{\infty}(X)}$ be an algebra homomorphism. 
Then the induced map $\map{\iota_*|_X}{(X,\Sigma)}{\Sp{\norm{X\iota}{}}{A}}$ is $\Sigma$-measurable if and only if 
$\iota A\subseteq\Meas{b}{X,\Sigma}$.
\end{prop}
\begin{proof}
\ma{Note that by Corollary~\ref{MCmpt},   $\tilde{\Sigma}$ forms a basis for the topology of $\xi_\Sigma X$. But since $\Sigma$ and subsequently $\tilde{\Sigma}$ are countably generated,  every Borel subset of $\xi_\Sigma X$  belongs to $\tilde{\Sigma}$, by Lemma \ref{cntbly-gen}.}
A basic open set of $\Sp{\norm{X\iota}{}}{A}$ is of the form $\hat{a}^{-1}(O)$ where $O\subseteq\fld$ is open. Computing the 
inverse image of $\hat{a}^{-1}(O)$ under $\iota_*$ we have:
\begin{equation}\label{Xopen}
	\iota_*|_X^{-1}\hat{a}^{-1}(O) = \hat{\iota a}^{-1}(O)\cap X
\end{equation}

($\Rightarrow$) Suppose that $\iota_*$ is $\Sigma$-measurable. If in contrary $\iota a\not\in\Meas{b}{X,\tau}$ for some $a\in A$,
then there exists a  set $O\subseteq\fld$, such that $\hat{\iota a}^{-1}(O)\cap X$ is not  $\Sigma$-measurable and hence by
\eqref{Xopen}, $\iota_*|_X$ can not be $\Sigma$-measurable, a contradiction.

($\Leftarrow$) If each $\iota a$ is $\Sigma$-measurable, then $\hat{\iota a}^{-1}(O)$ is $\Sigma$-measurable for any open $O\subseteq\fld$
and again by \eqref{Xopen}, $\iota_*|_X$ is $\Sigma$-measurable.
\end{proof}
\subsection{Measures on $(X,\Sigma)$ and $\xi_{\Sigma}X$}
Starting with a measurable structure $(X,\Sigma)$, we identified $X$ as an open dense subset of a totally disconnected compact space
$\xi_{\Sigma}X$ where every bounded $\Sigma$-measurable function on $X$ extends continuously to $\xi_{\Sigma}X$. This naturally leads one to
ask  about  the relation between measures on $(X,\Sigma)$ and $\xi_{\Sigma}X$.
\begin{prop}\label{meas-ext}
Let $\mu$ be a positive measure on $(X,\Sigma)$. Then $\mu$ extends to a Borel measure $^{\ast}\mu$ on $\xi_{\Sigma}X$, satisfying
\[
	\forall E\in\Sigma\quad^{\ast}\mu(\tilde{E})=\mu(E).
\]
\end{prop}
\begin{proof}
Define a linear functional $\map{L}{\Cnt{}{\xi_{\Sigma}X}}{\reals}$ by 
\[
	L(f)=\int_X f|_X~d\mu,\quad\forall f\in\Cnt{}{\xi_{\Sigma}X}.
\]
Clearly $L$ is positive and hence by Riesz representation theorem, there exists a Borel measure $^{\ast}\mu$ on $\xi_{\Sigma}X$ such that
\[
	L(f)=\int_{\xi_{\Sigma}X}f~d~^{\ast}\mu,\quad\forall f\in\Cnt{}{\xi_{\Sigma}X}.
\]
Note that for every $E\in\Sigma$, $^{\ast}\mu(\tilde{E})=\int\tilde{\chi}_E~d~^{\ast}\mu=L(\tilde{\chi}_E)=\int\chi_E~d\mu=\mu(E)$.
\end{proof}
Although the measure $^{\ast}\mu$ obtained in Proposition \ref{meas-ext} seems to be mainly supported on $X$, but in fact, the size of
$X\cap\supp{^{\ast}\mu}$ is rather small as it is pointed out in the following proposition.
\begin{prop}\label{meas-shift}
Let $\mu$ be a finite Borel measure on $\xi_{\Sigma}X$ and  $\Sigma$ contains all singletons. Then $X\cap\supp{\mu}$ is at most countable.
\end{prop}
\begin{proof}
By definition, a point $x\in\xi_{\Sigma}X$ belongs to $\supp{\mu}$ if and only if for every neighbourhood $U$ of $x$, $\mu(U)>0$. 
Every singleton $\{z\}$ for $z\in X$ is open in $\xi_{\Sigma}X$, thus for every $z\in X\cap\supp{\mu}$, $\mu(\{z\})>0$.
Since $\mu(\xi_{\Sigma}X)<\infty$, $X\cap\supp{\mu}$ can not be uncountable.
\end{proof}
\begin{crl}
\ma{Let $\mu$ be a positive  measure on $(X, \Sigma)$.} If $\mu(\{x\})=0$, for some $x\in X$, then $x\not\in\supp{^{\ast}\mu}$.
\end{crl}
\begin{proof}
Since $\{x\}\in\Sigma$ and $\mu(\{x\})=0$, $\chi_x\in\Meas{b}{X,\Sigma}$ and $\int_X\chi_xd\mu=0$. 
Thus $^{\ast}\mu(\{x\})=\int_{\xi_{\Sigma}X}\tilde{\chi}_{x}d~^{\ast}\mu=0$. But $\{x\}$ is open and hence $x\not\in\supp{^{\ast}\mu}$.
\end{proof}
\subsection{Borel Algebra of a topology}\label{nstrd}
Let $(X,\tau)$ be a $T_1$ topological space and denote by $\mathcal{B}_{\tau}$ its Borel algebra. Since the topology is $T_1$, singletons are
Borel and hence $\Meas{b}{X,\mathcal{B}_{\tau}}$ separates points of $X$.
Clearly the inclusion $\map{\iota}{\Cnt{b}{X,\tau}}{\Meas{b}{X,\mathcal{B}_{\tau}}}$ is continuous and hence 
$\map{\iota_{\ast}}{\xi_{\mathcal{B}_{\tau}}X}{\Sp{\|\cdot\|_X}{\Cnt{b}{X,\tau}}}$ is onto. Consequently, if $\tau$ is completely regular, then $\beta X$
is a continuous image of $\xi_{\mathcal{B}_{\tau}X}$ where $\beta X$ is the Stone-\v{C}ech compactification of $X$ (look at \cite[6.5]{Gil-Jer}). But $\iota_\ast$  cannot be injective unless $\tau$ is extremely disconnected, in which case 
$\mathcal{B}_{\tau}=\tau$ and hence $\xi_{\mathcal{B}_{\tau}}$ and $\beta$ are identical. It is  natural to ask if there is any relation between 
topological structure of $(X,\tau)$ and $\xi_{\mathcal{B}_{\tau}}X$.

Let $x\in X$ and $\nbhd{x}{\tau}$ be the family of open neighbourhoods of $x$ in $\tau$ and $\tilde{\mathcal{N}}_{\tau}(x)=\{\tilde{U}:U\in\nbhd{x}{\tau}\}$.
Define the \textit{halo} of $x$ in $\xi_{\mathcal{B}_{\tau}}X$ as 
\[
	h(x) := \bigcap\tilde{N}_{\tau}(x).
\]
The set $h(x)$ is compact and contains all points of $\xi_{\mathcal{B}_{\tau}}X$ that can not be distinguished from $x$ via the image of $\tau$.
If $\tau$ is Hausdorff, then for every $x\neq y\in X$, there are open sets $U_x, U_y\in\tau$ with $U_x\cap U_y=\emptyset$. 
Thus $\tilde{U}_x\cap\tilde{U}_y=\emptyset$, and therefore $h(x)\cap h(y)=\emptyset$.
\begin{prop}\label{Open-halos}
If $\tau$ is Hausdorff, then $h(x)$ is open if and only if $\{x\}$ is open in $(X, \tau)$.
\end{prop}
\begin{proof}
If $\{x\}$ is open, then $\{x\}\in\nbhd{x}{\tau}$. Since $\tilde{\{x\}}=\{x\}$, clearly, $x\in h(x)\subseteq\{x\}$.
Conversely, if $h(x)$ is open, then it is clopen and hence $h(x)=\tilde{E}\in\tilde{\mathcal{B}}_{\tau}$. If ${E}\neq\{x\}$, then $E$ contains
another point $y\in X$, $y\neq x$. Thus $y\in h(x)$ which implies that $h(x)\cap h(y)\neq\emptyset$, contradicting the above argument.
\end{proof}
Proposition \ref{Open-halos} can be read as $h(x)=\{x\}$ if and only if $\{x\}$ is open in $(X, \tau)$. The following shows how the 
compactness of a Borel subset of $(X,\tau)$ is reflected in $\xi_{\mathcal{B}_{\tau}}X$.
\begin{thm}\label{semi-robinson}
Let $Y\subseteq(X,\tau)$ be a Borel subspace. Then $Y$ is compact if and only if $\tilde{Y}\subseteq\bigcup_{y\in Y}h(y)$.
\end{thm}
\begin{proof}
($\Rightarrow$) Suppose that $Y$ is compact and let $z\in\xi_{\mathcal{B}_{\tau}}X\setminus\bigcup_{y\in Y}h(y)$. 
We show $z\not\in\tilde{Y}$. Since $z\not\in\bigcup_{y\in Y}h(y)$, for each $y\in Y$, there exists $O_y\in\mathcal{N}_{\tau}(y)$
such that $z\not\in\tilde{O}_y$. Now $\{O_y:y\in Y\}$ is an open cover of the compact set $Y$. Let $\{O_{y_1},\dots,O_{y_k}\}$ be
such that $Y\subseteq\bigcup_{i=1}^kO_{y_i}$, then $\tilde{Y}\subseteq\bigcup_{i=1}^k\tilde{O}_{y_i}$ which proves $z\not\in\tilde{Y}$,
and hence $\tilde{Y}\subseteq\bigcup_{y\in Y}h(y)$.

($\Leftarrow$) Suppose that $\tilde{Y}\subseteq\bigcup_{y\in Y}h(y)$, but $Y$ is not compact. Then there exists an open cover
$\{O_i\}_{i\in I}$ of $Y$ with no finite subcover. So, for every finite subset $\{i_1,\dots,i_n\}$ of $I$, 
\[
	Y\cap\left(\bigcap_{k=1}^nO_{i_k}\right)\neq\emptyset.
\]
Since $Y$ is Borel, $\tilde{Y}$ is compact and hence $\{\tilde{O}_i\}_{i\in I}$ forms a basis for an ultrafilter $\mathcal{F}$ in 
$\xi_{\mathcal{B}_{\tau}}X$. Clearly $\tilde{Y}\in\mathcal{F}$ and hence $z=\lim\mathcal{F}\in\tilde{Y}$ (for more detail on filters see
\cite[\S12]{willard}). For every $y\in Y$, there exists 
$i\in I$ such that $O_i\in\nbhd{y}{\tau}$ and hence $z\not\in\tilde{O}_i$. Thus $z\not\in h(y)\subseteq\tilde{O}_i$. This proves
\[
	z\in\tilde{Y}\setminus\bigcup_{y\in Y}h(y),
\]
as desired.
\end{proof}
It is worth mentioning that the results of \ref{nstrd} resemble significant similarities between properties of $\xi_{\mathcal{B}_{\tau}}X$
and nonstandard extensions of $(X,\tau)$. We can consider $\xi_{\mathcal{B}_{\tau}}X$ as a nonstandard model of $(X,\tau)$ and characterize 
halos as analogue to monads and so on. In this scope Theorem \ref{semi-robinson} is the analogue of Robinson's theorem 
\cite[Theorem III.2.2]{Hurd-Loeb} about nonstandard extension of compact spaces.

\section*{Acknowledgements}
For this research, the first author  was supported by a postdoctoral Fellowship at the University of Waterloo and a postdoctoral fellowship at 
Chalmers University of Technology and  University of Gothenburg.  The second author was supported by a   postdoctoral Fellowship at the University 
of Saskatchewan.   The authors thank  Reza Koushesh and Nico Spronk for their productive  comments on the first draft of this manuscript.   \ma{The authors are also grateful to the anonymous referee for making valuable comments which improved the exposition of the paper.} 
\bibliographystyle{plain}
\bibliography{Bibliography}
\end{document}